\documentclass[12pt,leqno]{article}
\usepackage{amssymb}
\usepackage[mathscr]{eucal}
\usepackage{amsmath,amssymb,latexsym,theorem,bbm}
\usepackage{color,url}
\usepackage{enumitem}

\newcommand{\SC}{\scriptstyle}

\newcommand{\CC}{\mathsf{C}}
\newcommand{\DD}{\mathsf{D}}

\newcommand{\NN}{\mathbb{N}}
\newcommand{\QQ}{\mathbb{Q}}
\newcommand{\RR}{\mathbb{R}}

\newcommand{\ZZ}{\mathbb{Z}}

\newcommand{\cA}{{\mathcal A}}
\newcommand{\cB}{{\mathcal B}}

\newcommand{\cD}{{\mathcal D}}

\newcommand{\cF}{{\mathcal F}}

\newcommand{\cU}{{\mathcal U}}

\newcommand{\bcU}{\boldsymbol{\cU}}

\newcommand{\cY}{{\mathcal Y}}

\newcommand{\bcY}{\boldsymbol{\cY}}

\newcommand{\dd}{\mathrm{d}}
\newcommand{\ee}{\mathrm{e}}

\newcommand{\slu}{{\SC\mathrm{lu}}}

\newcommand{\EE}{\operatorname{\mathbb{E}}}
\newcommand{\PP}{\operatorname{\mathbb{P}}}

\newcommand{\var}{\operatorname{Var}}

\renewcommand{\mid}{\,|\,}

\renewcommand{\leq}{\leqslant}
\renewcommand{\geq}{\geqslant}

\newcommand{\stoch}{\stackrel{\PP}{\longrightarrow}}
\newcommand{\distr}{\stackrel{\cD}{\longrightarrow}}
\newcommand{\distrf}{\stackrel{\cD_f}{\longrightarrow}}
\newcommand{\distre}{\stackrel{\cD}{=}}

\newcommand{\qmean}{\stackrel{L_2}{\longrightarrow}}

\newcommand{\lu}{\stackrel{\slu}{\longrightarrow}}
\newcommand{\Jto}{\stackrel{J_1}{\longrightarrow}}
\newcommand{\Jtoo}[1]{\stackrel{J_1(#1)}{\longrightarrow}}

\newcommand{\ase}{\stackrel{{\mathrm{a.s.}}}{=}}

\newcommand{\bbone}{\mathbbm{1}}
\newcommand{\ns}{{\lfloor ns\rfloor}}
\newcommand{\nt}{{\lfloor nt\rfloor}}

\newcommand{\nT}{{\lfloor nT\rfloor}}

\newcommand{\proofend}{\hfill\mbox{$\Box$}}

\numberwithin{equation}{section}

\theoremstyle{change} \theorembodyfont{\em}
\newtheorem{Lem}{Lemma.}[section]
\newtheorem{Thm}[Lem]{Theorem.}

\newtheorem{Cor}[Lem]{Corollary.}
\newtheorem{Def}[Lem]{Definition.}

\theorembodyfont{\rm}
\newtheorem{Rem}[Lem]{Remark.}
\newtheorem{Ex}[Lem]{Example.}

\begin{document}

\begin{center}
{\bfseries\Large
 Functional limit theorems for\\[-2mm] Galton--Watson processes with inhomogeneous\\[1mm]
 immigration towards a deterministic limit function}

\vspace*{3mm}

 {\sc\large
  M\'aty\'as $\text{Barczy}^{*}$ \text{and}
  \ D\'aniel $\text{Bezd\'any}^{**,\diamond}$ }

\end{center}

\vskip0.2cm

\noindent
 * HUN-REN--SZTE Analysis and Applications Research Group,
   Bolyai Institute, University of Szeged,
   Aradi v\'ertan\'uk tere 1, H--6720 Szeged, Hungary.

\noindent
 ** Bolyai Institute, University of Szeged,
    Aradi v\'ertan\'uk tere 1, H--6720 Szeged, Hungary.

\noindent E-mails: barczy@math.u-szeged.hu (M. Barczy),
                   bezdany@server.math.u-szeged.hu (D. Bezd\'any).

\noindent $\diamond$ Corresponding author.

\vskip0.2cm

\renewcommand{\thefootnote}{}

\vspace*{0.1cm}

\begin{abstract}
We study the asymptotic behavior of a sequence of Galton--Watson processes with inhomogeneous immigration
 when the limit of the means of the offspring distributions is equal to $1$.
Under a growth condition on the variances of the offspring distributions,
 assuming the locally uniform convergence of the expectation functions of
 properly scaled immigration processes towards a c\`adl\`ag function,
  and the locally uniform convergence of the mean absolute deviation functions of the scaled immigration processes in question towards zero,
 we establish a functional limit theorem for the sequence of Galton--Watson processes with inhomogeneous immigration in question
 towards a deterministic limit process.
We also provide some applications, among others, we show that a recent result of Rahimov and Sharipov (2024)
 on a functional limit theorem for a sequence of nearly critical Galton--Watson processes with dependent immigration
 can be recovered using our result.
\end{abstract}

%\tableofcontents

\section{Introduction}\label{Intro}

Functional limit theorems lie at the heart of studying stochastic processes.
In this paper, for Galton--Watson processes with inhomogeneous immigration,
 we derive such a result with a deterministic limit process.
Galton–Watson processes with or without immigration can describe the evolution of populations,
 namely, individuals reproduce independently according to a given offspring law and,
 in case of immigration, the arrival of immigrants in each generation occurs independently 
 of the population’s evolution according to a given immigration law.
In this paper, we consider a modification of Galton-Watson processes with immigration, namely,
 the case when the number of immigrating individuals in different generations are not necessarily independent or identically distributed
 (i.e., the immigration law can depend on the generation). 
In mathematical terms, with $\ZZ_+$ and $\NN$ denoting the set of non-negative and positive integers, respectively,
 let $(Y_k)_{k\in\ZZ_+}$ be a $\ZZ_+$-valued stochastic process,
and let $(X_k)_{k\in\ZZ_+}$ be a stochastic process given by
\begin{equation}\label{X_def}
X_{k}=\sum_{j=1}^{X_{k-1}}\xi_{k,j}+Y_{k-1}, \qquad
	\text{$k\in\NN$,}
\end{equation}
 with $X_0=0$ and using the convention $\sum_{j=1}^0:=0$,
 where $\{\xi_{k,j}: k,j\in\NN\}$ are independent,
 identically distributed $\ZZ_+$-valued random variables independent of $\{Y_k: k\in\ZZ_+\}$.
We call $(X_k)_{k\in\ZZ_+}$ a Galton--Watson process with inhomogeneous immigration (GWII process),
 and $(Y_k)_{k\in\ZZ_+}$ is called the immigration process of $(X_k)_{k\in\ZZ_+}$.
If, in addition, $Y_k$, $k\in\ZZ_+$, are independent and identically distributed,
 then we say that $(X_k)_{k\in\ZZ_+}$
 is a Galton--Watson process with immigration (GWI process).
Intuitively, we can interpret $X_0$ as the initial size of a population,
 and, for each $k\in\NN$, $X_k$ as the size of the $k$-th generation of a population,
 $\xi_{k,j}$ as the number of offspring of the $j$-th individual in the $(k-1)$-th generation,
 and $Y_{k-1}$ as the number of immigrants joining the $k$-th generation.

 Studying the asymptotic behavior of Galton–Watson processes with and without immigration has a long and vast history. 
 The nature of these limit theorems highly depends on whether the offspring mean of the branching process 
 in question is less than, equal to, or greater than one (subcritical, critical, or supercritical cases).
In this paper, we investigate the asymptotic behavior of a sequence of Galton–Watson processes
 with inhomogeneous immigration in the case where the limiting mean of the offspring distributions
 is equal to 1 (roughly speaking corresponding to the critical case).
For an overview on the earlier results on the asymptotic behavior of Galton–Watson processes with inhomogeneous immigration, 
 including some applications as well, see the introduction of our recent paper Barczy and Bezd\'any \cite{BarBez2}.
Here we only mention a book and a recent survey of Rahimov \cite{Rah3} and Rahimov \cite{Rah2}, respectively. 

The paper is organized as follows.
In Section \ref{Section_multi-type_GWI}, first we introduce some notations and hypotheses that are used throughout the paper, 
 and then we recall some results of Barczy and Bezd\'any \cite{BarBez2} on functional limit theorems 
 for some Galton–Watson processes with inhomogeneous immigration, see Theorems \ref{main_2} and \ref{exp_thm}.
In Section \ref{Section_conv_results}, first we derive a set of sufficient conditions 
 under which a sequence of properly scaled Galton–Watson processes with inhomogeneous immigration 
 converges in distribution to a deterministic limit process, see Theorem \ref{main_3}.
Then we present three remarks in which we highlight some connections among the conditions of Theorem \ref{main_3},
 see Remarks \ref{crit_rem}, \ref{Rem_d}, and \ref{Y_no_conv}.
We also provide several applications of Theorem \ref{main_3}.
In Example \ref{Ex_1}, we apply Theorem \ref{main_3} to a sequence of Poisson INteger-valued AutoRegressive processes of order $1$ 
 ($\mathrm{PINAR(1)}$ processes).
The application given in Example \ref{Ex_2} is motivated by Rahimov and Sharipov \cite[Example 1 and Corollary 1]{RahSha}, 
 where each immigration process is given by some (not necessarily independent) 
 Poisson distributed random variables such that their means are described in terms of
 a function regularly varying  at infinity with a positive index.
In Remark \ref{IspPapZui_rem}, we use Theorem \ref{main_3} to recover 
 Theorem 2.1 in Isp\'any et al.\ \cite{IspPapZui} in the case $b^2=0$ (using their notation).
We close Section \ref{Section_conv_results} by recalling Theorem 1 in Rahimov and Sharipov \cite{RahSha}
 in case of regularly varying functions with positive indices in their condition \textup{(\textbf{C1})},
 see Theorem \ref{RahSha_0}.
The reason for doing so is that, in Section \ref{RahSha_proof},
 we show that Theorem \ref{main_3} can be used to recover
 the above mentioned result of Rahimov and Sharipov \cite{RahSha}.
In Section \ref{Sec_preliminary}, we present some auxiliary notations 
 and results that are used in the proofs.
In particular, we recall some formulae for the conditional expectation of
 $X_k$ given $Y_0,X_0,\dots,Y_{k-1},X_{k-1}$, where $k\in\NN$.
Section \ref{Sec_Proof_Thm_main3} is devoted to the proof of Theorem \ref{main_3}.
In Section \ref{RahSha_proof}, as an application of Theorem \ref{main_3}, 
 we derive Theorem 1 in Rahimov and Sharipov \cite{RahSha} (see also Theorem \ref{RahSha_0}) by giving an alternative proof.
We close the paper with two appendices. 
Appendix \ref{GWII_moments_CMT} contains some results on locally uniform convergence,
 convergence in the Skorokhod $J_1$ topology, and a version of the continuous mapping theorem.
In Appendix \ref{App_reg}, we present some simple auxiliary results on regularly varying functions.

Finally, we emphasize that the method of the proof of our Theorem \ref{main_3} is 
 different from the one used in the proof of Theorem 1 in Rahimov and Sharipov \cite{RahSha}.
The proof of Theorem 1 in Rahimov and Sharipov \cite{RahSha} is based on a limit theorem which gives sufficient conditions under which the
 partial sum processes of a triangular array of dependent random variables converge in distribution to a
 deterministic continuous function, see, e.g., Rahimov and Sharipov \cite[Theorem A]{RahSha}.
More precisely, given a sequence $(X^{(n)}_k)_{k\in\ZZ_+}$, $n\in\NN$, of GWII processes 
 (with offspring and immigration random variables $\{\xi^{(n)}_{k,j}: k,j\in\NN\}$
  and $\{ Y^{(n)}_k : k\in\ZZ_+\}$, $n\in\NN$, respectively),
 they considered the decomposition  
 \[
  \frac{X^{(n)}_\nt}{(\EE(\xi^{(n)}_{1,1}))^{\nt}\EE(X_n^{(n)})} 
      = \sum_{k=1}^{\nt} U^{(n)}_k, \qquad n\in\NN, \; t\in\RR_+,
 \]
 where 
 \[
  U^{(n)}_k:=\frac{1}{\EE(X^{(n)}_n)} \big( (\EE(\xi^{(n)}_{1,1}))^{-k} X^{(n)}_k - (\EE(\xi^{(n)}_{1,1}))^{-k+1} X^{(n)}_{k-1} \big), \qquad n,k\in\NN,
 \]  
 and then they applied the above mentioned limit theorem for the triangular array of dependent random variables $\{U^{(n)}_k : n,k\in\NN\}$ to show that
\[
	\left(\sum_{k=1}^{\nt}U_k^{(n)}\right)_{t\in\RR_+}
     \quad \text{converges in distribution to}\quad
     \left(\frac{\int_0^t \ee^{-\gamma s}s^{\alpha_1}\,\dd s}{\int_0^1 \ee^{\gamma(1-s)}s^{\alpha_1}\,\dd s}\right)_{t\in\RR_+} \quad \text{as $n\to\infty$.}
\]
In the course of the proof they had to check, for example, an assumption
 on the sum of the conditional second moments of $U_k^{(n)}$, $n,k\in\NN$, namely that
 $\sum_{k=1}^{\nT} \EE( (U^{(n)}_k)^2 \mid X^{(n)}_1,\ldots, X^{(n)}_{k-1})$
    converges in probability to $0$ as $n\to\infty$ for all $T\in\RR_{++}$.
The above proof method of Rahimov and Sharipov \cite{RahSha}
 is different from the one used in the proof of our Theorem \ref{main_3}, since
 we consider a different decomposition of $X_\nt^{(n)}$ given by
\begin{align*}
X_\nt^{(n)}
		=&\sum_{j=1}^\nt\left(\EE(\xi_{1,1}^{(n)})\right)^{\nt-j}\EE\left(Y_{j-1}^{(n)}\right)
		+\sum_{j=1}^\nt\left(\EE(\xi_{1,1}^{(n)})\right)^{\nt-j}\left(Y_{j-1}^{(n)}-\EE\left(Y_{j-1}^{(n)}\right)\right)\\
		&+\sum_{j=1}^\nt\left(\EE(\xi_{1,1}^{(n)})\right)^{\nt-j}\left(X_j^{(n)}- \EE(\xi_{1,1}^{(n)})X_{j-1}^{(n)}-Y_{j-1}^{(n)}\right), \qquad n\in\NN,\, t\in\RR_+,
\end{align*}
see also \eqref{3_cr_X_decomp} and \eqref{X_decomp_3}.
Using a limit theorem for the expected value functions of GWII processes due to Barczy and Bezd\'any \cite[part \textup{(ii)} of Theorem 3.5]{BarBez2} (see also Theorem \ref{exp_thm}),
 we have that, after appropriate scaling, the sequence of stochastic processes formed from the first terms on the right hand
 side of the equality above converges in distribution to the limit process \eqref{main_3_X_conv} in Theorem \ref{main_3}.
To show that the remaining two terms become negligible as $n\to\infty$ with the same scaling,
 we use a kind of Slutsky's lemma for stochastic processes with c\`adl\`ag sample paths
 (see, e.g., Jacod and Shiryaev \cite[Chapter VI, Lemma 3.31]{JacShi}).

\section{Galton--Watson processes with inhomogeneous immigration}\label{Section_multi-type_GWI}

In this section, first we introduce some notations and hypotheses that are used throughout the paper, 
 and then we recall some results of Barczy and Bezd\'any \cite{BarBez2} on functional limit theorems 
 for some Galton–Watson processes with inhomogeneous immigration, see Theorems \ref{main_2} and \ref{exp_thm}.

Let $\ZZ_+$, $\NN$, $\QQ$, $\RR$, $\RR_+$ and $\RR_{++}$ denote the set
of non-negative integers, positive integers, rational numbers, real numbers, non-negative real
numbers and positive real numbers, respectively.
The maximum of two real numbers $x,y\in\RR$ is denoted by $x\vee y$.

Every random variable will be defined on a fixed probability space $(\Omega, \cA, \PP)$.
Convergence in probability, convergence in $L_2$, equality in distribution and almost sure equality
 is denoted by $\stoch$, $\qmean$, $\distre$ and $\ase$, respectively.
A function $f : \RR_+ \to \RR^d$ is called c\`adl\`ag if it is right continuous with left limits.
For $d\in\NN$, we will use $\DD(\RR_+,\RR^d)$ for the set of c\`adl\`ag functions from $\RR_+$ to $\RR^d$,
and $\CC(\RR_+,\RR^d)$ for the set of continuous functions from $\RR_+$ to $\RR^d$.
For $d\in\NN$, the weak convergence of the finite dimensional distributions of $\RR^d$-valued stochastic processes with sample paths in $\DD(\RR_+, \RR^d)$
and the weak convergence of $\RR^d$-valued stochastic processes with sample paths in $\DD(\RR_+, \RR^d)$
are denoted by $\distrf$ and $\distr$, respectively.
For functions $f:\RR_+\to\RR^d$ and $f_n:\RR_+\to\RR^d$, $n \in \NN$,
 we write $f_n \lu f$ as $n\to\infty$ if $f_n$ converges to $f$
 locally uniformly as $n\to\infty$, i.e., if $\sup_{t \in K\cap\RR_+} \|f_n(t) - f(t)\| \to 0$
 as $n \to \infty$ for all compact sets $K$ of $\RR$,
 or, equivalently, if $\sup_{t \in [0,T]} \|f_n(t) - f(t)\| \to 0$ as $n\to\infty$ for all $T\in\RR_+$
 (for some properties of locally uniform convergence, see Appendix \ref{GWII_moments_CMT}).
For functions $f\in\DD(\RR_+,\RR^d)$ and $f_n\in\DD(\RR_+,\RR^d)$, $n \in \NN$,
 we write $f_n\Jtoo{d}f$ as $n\to\infty$ if $f_n$ converges to $f$
 as $n\to\infty$ in the Skorokhod $J_1$ topology,
 and if the dimension is clear from context, simply $f_n \Jto f$ as $n\to\infty$
 (for details about Skorokhod $J_1$ topology, see Appendix \ref{GWII_moments_CMT}).

In this paper, we consider a sequence of GWII processes.
For each $n\in\NN$, let $(X_k^{(n)})_{k\in\ZZ_+}$ be a GWII process in the sense of \eqref{X_def}
with immigration process $(Y_k^{(n)})_{k\in\ZZ_+}$, that is,
\begin{equation}\label{Xn_def}
X_{k}^{(n)}=\sum_{j=1}^{X_{k-1}^{(n)}}\xi_{k,j}^{(n)}+Y_{k-1}^{(n)}, \qquad \text{$k\in\NN$ \ with \ $X_0^{(n)}=0$},\quad n\in\NN.
\end{equation}

For the rest of this paper,
\begin{equation*}
\text{
we assume that \ $m_n\in\RR_{++}$, $n\in\NN$, \ is a fixed sequence of positive real numbers.
}\end{equation*}

Next, we introduce four hypotheses.
Whenever we suppose that some of these hypotheses are satisfied, we will state them explicitly.
\begin{enumerate}[label=(\textbf{A\arabic*})]
\item\label{A1}
	We have that $\EE\big(\xi_{1,1}^{(n)}\big)<\infty$, $n\in\NN$.
\item\label{A2}
	We have that $\EE\big(Y_k^{(n)}\big)<\infty$, $k\in\ZZ_+$, $n\in\NN$,
	and there exists $L\in\DD(\RR_+,\RR)$ such that
	\begin{equation*}\label{E_conv}
	\limsup_{n\to\infty}\sup_{s\in[0,t]}\EE\big(m_n^{-1}Y_{\ns}^{(n)}\big)
	\leq L(t),	\qquad t\in\RR_+.
	\end{equation*}
\item\label{A3}
	There exists a $\RR_+$-valued stochastic process $(\cY_t)_{t\in\RR_+}$ with c\`adl\`ag sample paths such that
	\begin{equation*}
	\left(m_n^{-1}Y_\nt^{(n)}\right)_{t\in\RR_+}
	\distr
	\left(\cY_t\right)_{t\in\RR_+}
	\qquad \text{as $n\to\infty$.}
	\end{equation*}
\item\label{A4}
	We have that $\EE\big((\xi_{1,1}^{(n)})^2\big)<\infty$, $n\in\NN$, and $m_n^{-1}\var\big(\xi_{1,1}^{(n)}\big)\to0$ as $n\to\infty$.
\end{enumerate}

In the next remark, we formulate a simple observation regarding our hypotheses.

\begin{Rem}\label{Rem_1}
	Hypothesis \textup{\ref{A2}} holds if $\EE\big(Y_k^{(n)}\big)<\infty$, $k\in\ZZ_+$, $n\in\NN$,
	 and there exists a function $\beta\in\DD(\RR_+,\RR)$ such that
	\begin{equation}\label{A2_beta}
	\left(\EE\big(m_n^{-1}Y_\nt^{(n)}\big)\right)_{t\in\RR_+}\Jto\left(\beta(t)\right)_{t\in\RR_+} \qquad \text{as $n\to\infty$.}
	\end{equation}
	In this case, we may choose $L(t):=\sup_{s\in[0,t+1]}\beta(s)$, $t\in\RR_+$ (see Barczy and Bezd\'any \cite[part (i) of Remark 2.4]{BarBez2}).
	Note also that if $\beta$ is continuous, then, by part \textup{(i)} of Lemma \ref{Jto_basic},
	 \eqref{A2_beta} is equivalent to
	\begin{equation}\label{A2_beta_cont}
	\left(\EE(m_n^{-1}Y_\nt^{(n)})\right)_{t\in\RR_+}\lu\left(\beta(t)\right)_{t\in\RR_+} \qquad \text{as $n\to\infty$.}
	\end{equation}
\proofend
\end{Rem}

We recall two results from Barczy and Bezd\'any \cite[Theorems 3.3 and part (ii) of Theorem 3.5]{BarBez2}.

\begin{Thm}\label{main_2}
For each $n\in\NN$, let $(X_k^{(n)})_{k\in\ZZ_+}$ be given by \eqref{Xn_def},
and assume that \textup{\ref{A1}--\ref{A4}} hold.
Furthermore, assume that there exist $\gamma\in\RR$ and $\gamma_n\in\RR$, $n\in\NN$,
such that $\gamma_n\to\gamma$ as $n\to\infty$ and $\EE(\xi_{1,1}^{(n)})=1+\frac{\gamma_n}{n}$, $n\in\NN$.
Then
\begin{equation*}
\left(\begin{bmatrix}
	m_n^{-1}Y_\nt^{(n)}\\[1mm]
	(nm_n)^{-1}X_\nt^{(n)}
	\end{bmatrix}\right)_{t\in\RR_+}
\distr	\left(\begin{bmatrix}
	\cY_t\\
	\int_0^t\ee^{\gamma(t-s)}\cY_s\,\dd s
	\end{bmatrix}\right)_{t\in\RR_+}
 \qquad \text{as $n\to\infty$.}
\end{equation*}
\end{Thm}

The next result establishes a limit theorem for the expected value functions of the GWII processes in Theorem \ref{main_2}.

\begin{Thm}\label{exp_thm}
For each $n\in\NN$, let $(X_k^{(n)})_{k\in\ZZ_+}$ be given by \eqref{Xn_def},
and assume that \textup{\ref{A1}} holds, $\EE(Y_k^{(n)})<\infty$, $k\in\ZZ_+$, $n\in\NN$,
 and there exists a function $\beta\in\DD(\RR_+,\RR)$ such that \eqref{A2_beta} holds.
Furthermore, assume that there exist
	 $\gamma\in\RR$ and $\gamma_n\in\RR$, $n\in\NN$,
	 such that $\gamma_n\to\gamma$ as $n\to\infty$
	 and $\EE(\xi_{1,1}^{(n)})=1+\frac{\gamma_n}{n}$, $n\in\NN$.
Then
	\begin{equation*}\label{cr_exp}
	\left(
	(nm_n)^{-1}\EE\left(
	X_\nt^{(n)}\right)\right)_{t\in\RR_+}
		\lu
	\left(
	\int_0^t\ee^{\gamma(t-s)}\beta(s)\,\dd s\right)_{t\in\RR_+}
	 \qquad \text{as $n\to\infty$.}
	\end{equation*}
\end{Thm}

We mention that Theorem \ref{exp_thm} implies part (a) of Lemma 2 in Rahimov \cite{Rah1} for regularly varying functions at infinity
 with positive indices, see the proof of Theorem \ref{RahSha_0}. 

\section{Results}\label{Section_conv_results}

First, we derive a set of sufficient conditions (different from the one in Theorem \ref{main_2})
 under which $\big( (nm_n)^{-1} X^{(n)}_{\nt}\big)_{t\in\RR_+}$, $n\in\NN$,
 converges weakly to a deterministic limit process, see Theorem \ref{main_3}.
Comparing the assumptions of Theorem \ref{main_2} and those of Theorem \ref{main_3},
 we point out the facts that in Theorem \ref{main_3},
 instead of hypothesis \ref{A2} we suppose that \eqref{A2_beta} holds for a function $\beta\in\DD(\RR_+,\RR)$,
 and instead of hypothesis \ref{A3} we suppose that
 the expectation functions of
 $(m_n^{-1}|Y_\nt^{(n)}-\EE(Y_\nt^{(n)})|)_{t\in\RR_+}$, $n\in\NN$,
 converge to the constant zero function locally uniformly.
Furthermore, Theorem \ref{main_3} only asserts the weak convergence of
 the stochastic processes $((nm_n)^{-1}X_\nt^{(n)})_{t\in\RR_+}$, $n\in\NN$,
 not the weak convergence of the stochastic processes $([m_n^{-1}Y_\nt^{(n)},(nm_n)^{-1}X_\nt^{(n)}]^\top)_{t\in\RR_+}$, $n\in\NN$
 (in fact, it is possible that the stochastic processes $(m_n^{-1}Y_\nt^{(n)})_{t\in\RR_+}$, $n\in\NN$,
 do not converge weakly, see Remark \ref{Y_no_conv}).
  
\begin{Thm}\label{main_3}
For each $n\in\NN$, let $(X_k^{(n)})_{k\in\ZZ_+}$ be given by \eqref{Xn_def}.
Suppose that
\begin{itemize}
\item[(a)]
	hypotheses \textup{\ref{A1}} and \textup{\ref{A4}} hold,
\item[(b)]
	there exist
	 $\gamma\in\RR$ and $\gamma_n\in\RR$, $n\in\NN$,
	 such that $\gamma_n\to\gamma$ as $n\to\infty$ and $\EE(\xi^{(n)}_{1,1})=1+\frac{\gamma_n}{n}$, $n\in\NN$,
\item[(c)] $\EE(Y^{(n)}_k)<\infty$, $k\in\ZZ_+$, $n\in\NN$, and
	there exists a function $\beta\in\DD(\RR_+,\RR)$ such that \eqref{A2_beta} holds,
\item[(d)]
	it holds that
	\begin{equation}\label{Y_mom_zero}
	\left(\EE\left(\left|m_n^{-1}Y_\nt^{(n)}-\EE(m_n^{-1}Y_\nt^{(n)})\right|\right)\right)_{t\in\RR_+}
	\lu\left(0\right)_{t\in\RR_+} \qquad \text{as $n\to\infty$.}
	\end{equation}
\end{itemize}
Then
\begin{equation}\label{main_3_X_conv}
\left((nm_n)^{-1}X_\nt^{(n)}\right)_{t\in\RR_+}
\distr	\left(\int_0^t\ee^{\gamma(t-s)}\beta(s)\,\dd s\right)_{t\in\RR_+}
 \qquad \text{as $n\to\infty$.}
\end{equation}
\end{Thm}

The next remark highlights the connection between condition (b)
 in Theorem \ref{main_3} and the concept of nearly critical GWI processes.

\begin{Rem}\label{crit_rem}
The condition (b) on $\EE(\xi_{1,1}^{(n)})$, $n\in\NN$, in Theorem \ref{main_3}
implies that $\EE(\xi_{1,1}^{(n)})\to1$ as $n\to\infty$.
Furthermore, if $\EE(\xi_{1,1}^{(n)})=1$, $n\in\NN$,
 then $\gamma_n=0$, $n\in\NN$, $\gamma=0$, and
 the limit stochastic process in \eqref{main_3_X_conv} takes the form $(\int_0^t\beta(s)\,\dd s)_{t\in\RR_+}$.
If $(Y_k^{(n)})_{k\in\ZZ_+}$ is a sequence of i.i.d.\ $\ZZ_+$-valued random variables for each $n\in\NN$,
then this means that the corresponding sequence of GWI processes $(X_k^{(n)})_{k\in\ZZ_+}$, $n\in\NN$,
is nearly critical (see, e.g., Isp\'any \cite[Definition 1.1]{Isp}),
 and if, in addition, $\EE(\xi_{1,1}^{(n)})=1$, $n\in\NN$, then the corresponding
 sequence of GWI processes is critical (see, e.g., Athreya and Ney \cite[V.3]{AN}).
\proofend
\end{Rem}

Next, we provide two remarks.
In the first one, we point out the fact that condition (c) in Theorem \ref{main_3} in general does not
imply condition (d) in Theorem \ref{main_3}, and in the second one, we give an example for a sequence
of GWII processes for which the conditions (a)--(d) of Theorem \ref{main_3} hold, but condition \ref{A3} does not.

\begin{Rem}\label{Rem_d}
Let $U_n$, $n\in\NN$, be random variables such that 
$\PP(U_n=n)=n^{-1}$ and $\PP(U_n=0)=1-n^{-1}$, $n\in\NN$,
and define $Y^{(n)}_k:=U_n$, $k\in\ZZ_+$, $n\in\NN$, and $m_n:=1$, $n\in\NN$.
Then $\EE(m_n^{-1} Y^{(n)}_{\nt}) = \EE(U_n) = 1$, $n\in\NN$, $t\in\RR_+$,
yielding that condition (c) of Theorem \ref{main_3} holds with $\beta(t):=1$, $t\in\RR_+$.
On the other hand, we have that 
 \begin{align*}
  \EE\big(\big|m_n^{-1} Y^{(n)}_{\nt} - \EE(m_n^{-1} Y^{(n)}_{\nt})\big|\big) 
   = \EE( \vert U_n - 1\vert)
   = (n-1)n^{-1} + 1(1-n^{-1})
   = 2-2n^{-1}\to 2
 \end{align*}
 as $n\to\infty$, yielding that condition (d) of Theorem 3.1 does not hold.
\proofend
\end{Rem}

\begin{Rem}\label{Y_no_conv}
For each $n\in\NN$, let $U_n$ be a random variable uniformly distributed on $\{0,\dots,n-1\}$.
For each $n\in\NN$, let $m_n:=n$, and let $Y_k^{(n)}=n\cdot\bbone_{\{U_n=k\}}$, $k\in\ZZ_+$, that is,
\[
Y_k^{(n)}=\begin{cases}
n, & \text{if $U_n=k$,}\\
0, & \text{otherwise,}
\end{cases}\qquad k\in\ZZ_+.
\]
Then conditions \textup{(c)} and \textup{(d)} of Theorem \ref{main_3} are satisfied with the function $\beta(t):=0$, $t\in\RR_+$.
Indeed, we have that $\EE(Y_k^{(n)})=0\cdot\PP(U_n\neq k)+n\cdot\PP(U_n=k)$, $k\in\ZZ_+$, $n\in\NN$,
 and thus
\[
\EE(Y_k^{(n)})=n\cdot\PP(U_n=k)
	=\begin{cases}
n\cdot\frac{1}{n}=1, & \text{if $k\in\{0,\dots,n-1\}$,}\\
n\cdot 0=0, & \text{otherwise,}
\end{cases}\qquad n\in\NN.
\]
Consequently, for all $T\in\RR_{++}$, we have that
\[
\sup_{t\in[0,T]}\EE(|n^{-1}Y_\nt^{(n)}-0|)
	=n^{-1}\max_{k\in\{0,\dots,\nT\}}\EE(Y_k^{(n)})
	\leq n^{-1}\cdot1\to0 \qquad \text{as $n\to\infty$,}
\]
 which shows that \eqref{A2_beta_cont} (and thus also \eqref{A2_beta})
 holds with the given function $\beta$, thus condition \textup{(c)} is satisfied.
Furthermore, for all $T\in\RR_{++}$, we have that
\[
\sup_{t\in[0,T]}\EE(|n^{-1}Y_\nt^{(n)}-\EE(n^{-1}Y_\nt^{(n)})|)
	\leq 2n^{-1}\sup_{t\in[0,T]}\EE(Y_\nt^{(n)})
	\leq2n^{-1}\to0 \qquad \text{as $n\to\infty$,}
\]
which shows that condition \textup{(d)} is satisfied.
Consequently, if $\xi_{1,1}^{(n)}$, $n\in\NN$,
 are such that the conditions \textup{(a)} and \textup{(b)} are satisfied,
 then for the corresponding GWII processes $(X_k^{(n)})_{k\in\ZZ_+}$, $n\in\NN$,
 by Theorem \ref{main_3}, we get that $(n^{-2}X_\nt^{(n)})_{t\in\RR_+}\distr(0)_{t\in\RR_+}$ as $n\to\infty$.

Next, we show that there exists no stochastic process $(\cY_t)_{t\in\RR_+}$ with c\`adl\`ag sample paths
 such that condition \ref{A3} holds.
Note that, for all $t\in\RR_+$, we have that
\[
\EE(m_n^{-1}Y_\nt^{(n)})=n^{-1}\EE(Y_\nt^{(n)})\leq n^{-1}\to0 \qquad \text{as $n\to\infty$,}
\]
 and thus $m_n^{-1}Y_\nt^{(n)}\distr0$ as $n\to\infty$.
A similar argument yields that $(m_n^{-1}Y_\nt^{(n)})_{t\in\RR_+}\distrf(0)_{t\in\RR_+}$ as $n\to\infty$.
On the contrary, suppose that there exists a stochastic process
 $(\cY_t)_{t\in\RR_+}$ such that \ref{A3} holds.
Then the condition \textup{(9)} on the modified modulus of continuity
 of $(m_n^{-1}Y_\nt^{(n)})_{t\in\RR_+}$, $n\in\NN$,
 of Theorem 14.10 in Kallenberg \cite{Kal} holds.
Together with the weak convergence of the finite dimensional distributions shown previously,
 by Kallenberg \cite[Theorem 14.10]{Kal}, this yields that
 $(m_n^{-1}Y_\nt^{(n)})_{t\in\RR_+}\distr(0)_{t\in\RR_+}$ as $n\to\infty$,
 and hence $(\cY_t)_{t\in\RR_+}$ is the identically zero process.
On the other hand, recall that, by Jacod and Shiryaev \cite[Chapter VI, Proposition 2.4]{JacShi},
 the mapping $\DD(\RR_+,\RR)\ni f\mapsto\sup_{t\in[0,1]}|f(t)|\in\RR$
 is continuous on $\CC(\RR_+,\RR)$.
Consequently, since the identically zero process has continuous sample paths,
 by the continuous mapping theorem (see, e.g., Billingsley \cite[Chapter 1, Theorem 2.7]{Bil}),
 we get that
\begin{equation}\label{Y_no_conv_sup_conv}
\sup_{t\in[0,1]}\left|m_n^{-1}Y_\nt^{(n)}\right|\distr\sup_{t\in[0,1]}\left|\cY_t\right| \qquad \text{as $n\to\infty$.}
\end{equation}
However, we have that $\sup_{t\in[0,1]}\left|\cY_t\right|=0$ and
\[
\sup_{t\in[0,1]}\left|m_n^{-1}Y_\nt^{(n)}\right|
	=n^{-1}\max_{k\in\{0,\dots,n\}}Y_k^{(n)}
	=n^{-1}\max_{k\in\{0,\dots,n-1\}}Y_k^{(n)}
	=n^{-1} n
	=1, \qquad n\in\NN,
 \]
which shows that \eqref{Y_no_conv_sup_conv} cannot hold,
 yielding that there exists no stochastic process $(\cY_t)_{t\in\RR_+}$ such that \ref{A3} holds.
Consequently, there exists a sequence of GWII processes such that the conditions of Theorem \ref{main_3} hold, but not those of Theorem \ref{main_2}.
\proofend
\end{Rem}

In the next two examples, we present two applications of Theorem \ref{main_3}.
In the first example, the function $\beta$ in condition \textup{(c)} of Theorem \ref{main_3}
 is a constant function, while in the second example, it is a constant multiple of a power function.
 
\begin{Ex}\label{Ex_1}
Let $\gamma\in\RR\setminus\RR_+$ and $\lambda\in\RR_+$,
 and for each $n\in\NN$, let $\gamma_n\in\RR$ and $\lambda_n\in\RR_{++}$
 be such that $\gamma_n\to\gamma$ and $\frac{\lambda_n}{n}\to\lambda$ as $n\to\infty$. 
Note that, without loss of generality, we can assume that $1+\frac{\gamma_n}{n}\in(0,1)$, $n\in\NN$.
For each $n\in\NN$, let $\{Y_k^{(n)}: k\in\ZZ_+\}$ be i.i.d.\ Poisson distributed random variables with parameter $\lambda_n$,
 and let $\{\xi_{k,j}^{(n)}: k,j\in\NN \}$ be i.i.d.\ Bernoulli distributed random variables with parameter $1+\frac{\gamma_n}{n}$.
For each $n\in\NN$, let $(X_k^{(n)})_{k\in\ZZ_+}$ be a GWII process defined in \eqref{Xn_def} with the previously introduced random variables.
Note that for each $n\in\NN$, $(X_k^{(n)})_{k\in\ZZ_+}$ is a Poisson INteger-valued AutoRegressive process of order $1$
 ($\mathrm{PINAR(1)}$ process).
We check that the GWII processes $(X_k^{(n)})_{k\in\ZZ_+}$, $n\in\NN$, satisfy the conditions of Theorem \ref{main_3}
 with $m_n:=n$, $n\in\NN$, and $\beta:\RR_+\to\RR_+$, $\beta(t):=\lambda$, $t\in\RR_+$.
Conditions \textup{(a)} and \textup{(b)} are obviously satisfied.
We have that $\EE(m_n^{-1}Y_{\nt}^{(n)}) = n^{-1}\EE(Y^{(n)}_0)= n^{-1}\lambda_n$, $n\in\NN$, $t\in\RR_+$, and, since  $n^{-1}\lambda_n\to\lambda\in\RR_+$ as $n\to\infty$,
 we have that condition \textup{(c)} holds with the given $\beta$.
Lastly, we have that $\var(m_n^{-1}Y_{\nt}^{(n)}) = n^{-2}\var(Y_0^{(n)}) =n^{-2}\lambda_n$, $n\in\NN$, $t\in\RR_+$, 
 where $n^{-2}\lambda_n\to 0$ as $n\to\infty$, and hence, by Jensen's inequality, we have
 \[
   \sup_{t\in[0,T]}\left(\EE\left(\left|m_n^{-1}Y_{\nt}^{(n)}-\EE(m_n^{-1}Y_{\nt}^{(n)})\right|\right)\right)^2
      	\leq\sup_{t\in[0,T]}  \var(m_n^{-1}Y_{\nt}^{(n)}) = n^{-2}\lambda_n \to0 \qquad \text{as $n\to\infty$}
 \]
 for all $T\in\RR_{++}$. 
Therefore, condition (d) holds as well.
Consequently, by Theorem \ref{main_3}, we have that
\[
\left(n^{-2} X_\nt^{(n)}\right)_{t\in\RR_+}
	\distr\left(\int_0^t\ee^{\gamma(t-s)}\lambda\,\dd s\right)_{t\in\RR_+}
	=\left(\frac{\lambda}{\gamma}(\ee^{\gamma t}-1)\right)_{t\in\RR_+}
		\qquad \text{as $n\to\infty$.}
\]
\proofend
\end{Ex}

\begin{Ex}\label{Ex_2}
Next, we provide an example motivated by Rahimov and Sharipov \cite[Example 1 and Corollary 1]{RahSha}.
Let $f:\RR_{++}\to\RR_{++}$ be a function regularly varying at infinity with index $\alpha\in\RR_{++}$.
Let $\lambda\in\RR_{+}$ and $\lambda_n\in\RR_{++}$, $n\in\NN$, be such that $\lambda_n\to\lambda$ as $n\to\infty$.
Let $\gamma\in\RR\setminus\RR_+$ and $\gamma_n\in\RR$, $n\in\NN$, be such that $\gamma_n\to\gamma$ as $n\to\infty$.
Note that, without loss of generality, we can assume that $1+\frac{\gamma_n}{n}\in(0,1)$, $n\in\NN$.
For each $n\in\NN$, let $\{Y_k^{(n)}: k\in\ZZ_+\}$ be (not necessarily independent)
 Poisson distributed random variables with $\EE(Y_k^{(n)})=\lambda_nf(k+1)$, $k\in\ZZ_+$.
For each $n\in\NN$, let $\{\xi_{k,j}^{(n)}: k,j\in\NN \}$ be i.i.d.\ $\ZZ_+$-valued random variables with finite second moments,
 such that $\EE(\xi_{1,1}^{(n)})= 1+\frac{\gamma_n}{n}$, $n\in\NN$, and $(f(n+1))^{-1}\var(\xi_{1,1}^{(n)})\to0$ as $n\to\infty$.
For each $n\in\NN$, let $(X_k^{(n)})_{k\in\ZZ_+}$ be a GWII process defined in \eqref{Xn_def} with the previously introduced random variables.
We check that the GWII processes $(X_k^{(n)})_{k\in\ZZ_+}$, $n\in\NN$, satisfy the conditions of Theorem \ref{main_3}
 with $m_n:=f(n+1)$, $n\in\NN$, and $\beta:\RR_+\to\RR_+$, $\beta(t):=\lambda t^\alpha$, $t\in\RR_+$.
Conditions \textup{(a)} and \textup{(b)} are obviously satisfied.
We have that $\EE(m_n^{-1}Y_{\nt}^{(n)}) =\lambda_n (f(n+1))^{-1}f(\nt+1)$, $t\in\RR_+$, $n\in\NN$, and thus,
 by Corollary \ref{reg_tf_cor} and part \textup{(ii)} of Lemma \ref{Jto_basic},
 we have that \eqref{A2_beta_cont} holds with the given function $\beta$.
Consequently, by part (i) of Lemma \ref{Jto_basic},  condition \textup{(c)} holds with the given function $\beta$.
Lastly, we have that $\var(m_n^{-1}Y_{\nt}^{(n)}) =\lambda_n (f(n+1))^{-2}f(\nt+1)$, $t\in\RR_+$, $n\in\NN$, 
 where, since $\alpha\in\RR_{++}$, we have that $f(n+1)\to\infty$ as $n\to\infty$
 (see, e.g., Bingham et al.\ \cite[Proposition 1.5.1]{BinGolTeu}),
 and thus $\lambda_n (f(n+1))^{-1}\to0$ as $n\to\infty$.
Therefore, by Corollary \ref{reg_tf_cor} and part (ii) of Lemma \ref{Jto_basic},
 we have that
\[
(\var(m_n^{-1}Y_\nt^{(n)}))_{t\in\RR_+}
	=\left(\frac{\lambda_n}{f(n+1)}\frac{f(\nt+1)}{f(n+1)}\right)_{t\in\RR_+}
	\lu(0\cdot t^\alpha)_{t\in\RR_+}=(0)_{t\in\RR_+} \qquad \text{as $n\to\infty$.}
\]
Hence, by Jensen's inequality, we have
 \[
   \sup_{t\in[0,T]}\left(\EE\left(\left|m_n^{-1}Y_{\nt}^{(n)}-\EE(m_n^{-1}Y_{\nt}^{(n)})\right|\right)\right)^2
      	\leq\sup_{t\in[0,T]}  \var(m_n^{-1}Y_{\nt}^{(n)})\to0 \qquad \text{as $n\to\infty$}
 \]
 for all $T\in\RR_{++}$, thus condition (d) holds as well.
Consequently, by Theorem \ref{main_3}, we have that
\[
\left(\left(n f(n+1)\right)^{-1} X_\nt^{(n)}\right)_{t\in\RR_+}
	\distr\left(\lambda\int_0^t\ee^{\gamma(t-s)}s^\alpha\,\dd s\right)_{t\in\RR_+}
		\qquad \text{as $n\to\infty$.}
\]
We mention that the previous argument also works when $f$
 is replaced by the identically 1 function (which is a particular slowly varying function),
$\lambda=0$, and $\beta(t)=0$, $t\in\RR_+$.
Indeed, let the random variables $\{Y_k^{(n)}: k\in\ZZ_+, n\in\NN\}$
 and $\{\xi_{k,j}^{(n)}: k,j,n\in\NN \}$ be defined as before,
 so that our assumptions hold with $\beta(t):=0$, $t\in\RR_+$, $\lambda:=0$,
 and with $g$ in place of $f$, where $g:\RR_{++}\to\RR_{++}$, $g(t):=1$, $t\in\RR_{++}$,
 that is, $\lambda_n\to0$ as $n\to\infty$,
 $\EE(Y_k^{(n)})=\lambda_n$, $k\in\ZZ_+$, $n\in\NN$,
 and $\var(\xi_{1,1}^{(n)})\to0$ as $n\to\infty$.
Next we verify that the GWII processes $(X_k^{(n)})_{k\in\ZZ_+}$, $n\in\NN$,
 defined in \eqref{Xn_def} with the previously introduced random variables
 satisfy the conditions of Theorem \ref{main_3} with $m_n:=g(n+1)=1$, $n\in\NN$.
As before, conditions \textup{(a)} and \textup{(b)} hold trivially,
 and since $\EE(m_n^{-1}Y_\nt^{(n)})=\EE(Y_0^{(n)})=\lambda_n\to0$
 as $n\to\infty$ for all $t\in\RR_+$, we have that \eqref{A2_beta_cont}
 (and thus also \eqref{A2_beta}) holds
 with the given function $\beta$,
 and thus condition \textup{(c)} holds with the given function $\beta$.
Furthermore, since $\var(m_n^{-1}Y_\nt^{(n)})=\var(Y_0^{(n)})=\lambda_n\to0$
 as $n\to\infty$ for all $t\in\RR_+$,
 by the same argument as before, condition \textup{(d)} also holds.
Consequently, by Theorem \ref{main_3}, we get that
 \begin{align}\label{help_6}
   \big( n^{-1} X^{(n)}_{\nt}\big)_{t\in\RR_+}
       \distr (0)_{t\in\RR_+}
    \qquad \text{as $n\to\infty$.}
 \end{align}
In particular, if $\{\xi_{k,j}^{(n)}: k,j\in\NN \}$ are i.i.d.\ Bernoulli distributed random variables with parameter $1+\frac{\gamma_n}{n}$ for each $n\in\NN$,
 then $\var(\xi^{(n)}_{1,1})= - \frac{\gamma_n}{n} \left(1+\frac{\gamma_n}{n}\right)\to 0$ as $n\to\infty$,  as desired, and hence \eqref{help_6} holds
 with the previously introduced GWII processes $(X^{(n)}_k)_{k\in\ZZ_+}$, $n\in\NN$.
Furthermore, if, in addition, $\{Y_k^{(n)}: k\in\ZZ_+\}$ are independent as well for each $n\in\NN$,
 then the conditions of Example \ref{Ex_1} are satisfied
 with $\lambda=0$, since $\lambda_n\to0$ as $n\to\infty$
 implies that $\frac{\lambda_n}{n}\to0$ as $n\to\infty$.
 Consequently, Example \ref{Ex_1} yields that
 $(n^{-2}X_\nt^{(n)})_{t\in\RR_+}\distr(0)_{t\in\RR_+}$ as $n\to\infty$.
We point out that in this case, \eqref{help_6} also implies the assertion of Example \ref{Ex_1}.
Indeed, since $n^{-1}\to0$ as $n\to\infty$,
 part \textup{(iii)} of Lemma \ref{Jto_basic},
 Lemmas \ref{cont_cont} and \ref{Conv2Funct},
 together with \eqref{help_6} imply that
 $(n^{-2}X_\nt^{(n)})_{t\in\RR_+}
	=(n^{-1}\cdot n^{-1} X^{(n)}_{\nt})_{t\in\RR_+}
       \distr (0\cdot0)_{t\in\RR_+}=(0)_{t\in\RR_+}$
	as $n\to\infty$
 (for a similar argument with more details,
 see the end of the proof of Theorem \ref{RahSha_0} in Section \ref{RahSha_proof}),
 where the limit is exactly $\big(\frac{\lambda}{\gamma}(\ee^{\gamma t}-1)\big)_{t\in\RR_+}$ with $\lambda=0$, which is the limit process in Example \ref{Ex_1}.
\proofend
\end{Ex}

In the next remark, we point out the fact that Theorem \ref{main_3} can be used to recover 
 Theorem 2.1 in Isp\'any et al.\ \cite{IspPapZui} in the case $b^2=0$ (using their notation).

\begin{Rem}\label{IspPapZui_rem}
For each $n\in\NN$, let $\{Y_k^{(n)}: k\in\ZZ_+\}$ be
  i.i.d.\ $\ZZ_+$-valued random variables with $\EE((Y_0^{(n)})^2)<\infty$,
 and let $\{\xi_{k,j}^{(n)}: k,j\in\NN\}$ be i.i.d.\ $\ZZ_+$-valued random variables
 with $\EE((\xi_{1,1}^{(n)})^2)<\infty$.
Suppose that there exist $\gamma\in\RR$ and $\gamma_n\in\RR$, $n\in\NN$,
 such that $\gamma_n\to\gamma$ as $n\to\infty$
 and $\EE(\xi^{(n)}_{1,1})=1+\frac{\gamma_n}{n}$, $n\in\NN$.
In addition, suppose that there exists $\lambda\in \RR_+$
 such that $\EE(Y_0^{(n)})\to\lambda$ as $n\to\infty$,
 and we have that $\var(\xi^{(n)}_{1,1})\to0$ as $n\to\infty$
 and $\var(Y_0^{(n)})\to0$ as $n\to\infty$.
For each $n\in\NN$, let $(X_k^{(n)})_{k\in\ZZ_+}$ be the GWII process
 defined in \eqref{Xn_def} with the previously introduced random variables.
We check that the GWII processes $(X_k^{(n)})_{k\in\ZZ_+}$, $n\in\NN$,
 satisfy the conditions of Theorem \ref{main_3}
 with $m_n:=1$, $n\in\NN$, and $\beta:\RR_+\to\RR_+$, $\beta(t):=\lambda$, $t\in\RR_+$.
Conditions \textup{(a)} and \textup{(b)} are obviously satisfied.
We have that $\EE(m_n^{-1}Y_{\nt}^{(n)}) =\EE(Y_0^{(n)})$, $t\in\RR_+$, $n\in\NN$,
 and thus, $\sup_{t\in\RR_+}|\EE(Y_\nt^{(n)})-\beta(t)|=|\EE(Y_0^{(n)})-\lambda|\to0$ as $n\to\infty$,
 which shows that $(\EE(Y_\nt^{(n)}))_{t\in\RR_+}\lu (\lambda)_{t\in\RR_+}$ as $n\to\infty$.
Hence, by part (i) of Lemma \ref{Jto_basic}, we have that
 condition \textup{(c)} holds with the given function $\beta$.
Condition \textup{(d)} follows from
\[
\sup_{t\in\RR_+}\EE\left(\left|Y_\nt^{(n)}-\EE(Y_\nt^{(n)})\right|\right)
	\leq\sup_{t\in\RR_+}\sqrt{\EE\left(\left|Y_\nt^{(n)}-\EE(Y_\nt^{(n)})\right|^2\right)}
	=\sqrt{\var\left(Y_0^{(n)}\right)}\to0
\]
 as $n\to\infty$, where we used Cauchy-Schwarz inequality.
Consequently, by Theorem \ref{main_3}, we have that
\[
\left(n^{-1}X_\nt^{(n)}\right)_{t\in\RR_+}
	\distr
	\left(\int_0^t\ee^{\gamma(t-s)}\beta(s)\,\dd s\right)_{t\in\RR_+}
	=
	\left(\lambda\int_0^t\ee^{\gamma u}\,\dd u\right)_{t\in\RR_+}
	\qquad \text{as $n\to\infty$,}
\]
 which is nothing else but Theorem 2.1 in Isp\'any et al.\ \cite{IspPapZui} in the case $b^2=0$ (in their notation).
\proofend
\end{Rem}

As an additional application, in Section \ref{RahSha_proof},
 we will show that Theorem \ref{main_3} can be used to
 derive Theorem 1 in Rahimov and Sharipov \cite{RahSha}
 in case of $\alpha_1,\alpha_2\in\RR_{++}$ (using the notation of Theorem \ref{RahSha_0})
 which we recall below.
For completeness, we mention that this result of Rahimov and Sharipov \cite{RahSha}
 holds also in the case when $\alpha_1=0$ or $\alpha_2=0$.
 
\begin{Thm}[Rahimov and Sharipov {\cite[Theorem 1]{RahSha}}]\label{RahSha_0}
For each $n\in\NN$, let $(X_k^{(n)})_{k\in\ZZ_+}$ be given by \eqref{Xn_def}.
Assume that $\EE((\xi_{1,1}^{(n)})^2)<\infty$, $n\in\NN$,
 and $\EE((Y_k^{(n)})^2)<\infty$, $k\in\ZZ_+$, $n\in\NN$.
Let $\alpha_1,\alpha_2\in\RR_{++}$, and let $R_1,R_2:\RR_{++}\to\RR_{++}$
 be functions regularly varying at infinity
 with indices $\alpha_1$ and $\alpha_2$, respectively.
Suppose that the following conditions hold.
\begin{enumerate}[label=(\textbf{C\arabic*})]
\item
	For all $t\in\RR_+$, we have that
	\[
	\frac{1}{R_1(n+1)}\max_{0\leq k\leq \nt}\left|\EE(Y_k^{(n)})-R_1(k+1)\right|\to0 \qquad \text{as $n\to\infty$,}
	\]
	and
	\[
	\frac{1}{R_2(n+1)}\max_{0\leq k\leq \nt}\left|\var(Y_k^{(n)})-R_2(k+1)\right|\to0 \qquad \text{as $n\to\infty$.}
	\]
\item
	There exist $\gamma\in\RR$ and $\gamma_n\in\RR$, $n\in\NN$,
	 such that $\gamma_n\to\gamma$ as $n\to\infty$
	 and $\EE(\xi_{1,1}^{(n)})=1+\frac{\gamma_n}{n}$, $n\in\NN$.
\item
	We have that $\frac{1}{R_1(n+1)}\var(\xi_{1,1}^{(n)})\to0$ as $n\to\infty$.
\item
	We have that $R_1(n)\to\infty$ and $\frac{1}{(R_1(n))^2}R_2(n)\to0$ as $n\to\infty$.
\end{enumerate}
Then $\EE(X_n^{(n)})>0$ for sufficiently large $n\in\NN$, and
\[
	\left(\frac{X_\nt^{(n)}}{\EE(X_n^{(n)})}\right)_{t\in\RR_+}
	\distr
	\left(\frac{\int_0^t\ee^{\gamma(t-s)}s^{\alpha_1}\,\dd s}{\int_0^1\ee^{\gamma(1-s)}s^{\alpha_1}\,\dd s}\right)_{t\in\RR_+}
	\qquad \text{as $n\to\infty$.}
\]
\end{Thm}

\section{Preliminaries for the proofs}\label{Sec_preliminary}

In this section, we present some auxiliary notations and results which will be useful in the proofs.
To improve readability, we will state the relevant facts
 only for the GWII process defined in \eqref{X_def},
 although they hold true for each GWII process defined in \eqref{Xn_def} as well
 after a few natural changes (e.g., writing $a_n$ instead of $a$, where the notations $a$ and $a_n$, $n\in\NN$, are introduced below in \eqref{Notation}).
For the same reason, we will only explicitly define
 the martingale differences \eqref{M_def} for the GWII process given in \eqref{X_def}.
In the proofs, we will denote by $(M_k^{(n)})_{k\in\ZZ_+}$, $n\in\NN$,
 the martingale differences corresponding to the GWII processes defined in \eqref{Xn_def},
defined analogously to \eqref{M_def}.

Let $(X_k)_{k\in\ZZ_+}$ be the GWII process defined in \eqref{X_def}, with $(Y_k)_{k\in\ZZ_+}$ its immigration process,
such that $\EE(\xi_{1,1}^2)<\infty$ and $\EE(Y_k)<\infty$, $k\in\ZZ_+$.
We will often make use of the natural filtration of the stochastic process $([Y_k,X_k])_{k\in\ZZ_+}$, which we denote by $(\cF_k)_{k\in\ZZ_+}$, where
\begin{equation}\label{filtration}
\cF_k:=\sigma\left(Y_0,X_0,\dots,Y_k,X_k\right), \qquad k\in\ZZ_+.
\end{equation}
Introduce the notation
\begin{equation}\label{Notation}
 a:=\EE\left(\xi_{1,1}\right)\in\RR_+.
 \end{equation}
When dealing with the GWII processes defined in \eqref{Xn_def}, for each $n\in\NN$,
 we denote the filtration corresponding to \eqref{filtration} by $(\cF_k^{(n)})_{k\in\ZZ_+}$,
 and we use the notation $a_n:=\EE(\xi_{1,1}^{(n)})$, provided that this quantity exists and is finite.

First, we introduce the martingale differences given by
\begin{equation}\label{M_def}
M_k:=
	X_k-\EE(X_k\mid\cF_{k-1}), \qquad k\in\NN,
\end{equation}
 with $M_0:=0$.
By \eqref{X_def}, using the tower rule, we get that
\begin{equation*}
    \EE(X_k\mid\cF_{k-1})
	=\EE\left(\sum_{j=1}^{X_{k-1}}\xi_{k,j}+Y_{k-1}\,\bigg|\,\cF_{k-1}\right)
	= \EE(\xi_{1,1}) X_{k-1}+Y_{k-1}, \qquad k\in\NN.
\end{equation*}
Hence, using \eqref{M_def} and \eqref{Notation}, we have that
\begin{equation*}\label{M_X}
M_k
	=X_k-aX_{k-1}-Y_{k-1}, \qquad k\in\NN.
\end{equation*}
Consequently, using $X_0=0$ and \eqref{M_def}, we obtain that
\begin{align}\label{X_alt}
\begin{split}
X_k
	&=aX_{k-1}+(M_k+Y_{k-1})
	=a^k X_0+\sum_{j=1}^ka^{k-j}\left(M_j+Y_{j-1}\right)\\
	&=\sum_{j=1}^ka^{k-j}Y_{j-1}+\sum_{j=1}^ka^{k-j}M_j
	=\sum_{j=0}^{k-1}a^{k-1-j}Y_{j}+\sum_{j=1}^ka^{k-j}M_j,
	\qquad k\in\NN.
\end{split}
\end{align}
It is easy to see that these martingale differences are uncorrelated,
since for each $i,j\in\NN$, $i<j$, we have
\begin{equation*}%\label{M_uncorr}
\EE(M_iM_j)
	=\EE(\EE(M_iM_j\mid\cF_{j-1}))
	=\EE(M_i\EE(M_j\mid\cF_{j-1}))
	=\EE(M_i0)=0.
\end{equation*}

Finally, in the next lemma, we present some formulae for the conditional expectation and the conditional variance of $X_k$ and $M_k$
 given $\cF_{k-1}$, respectively, where $k\in\NN$
 (for a proof, see the proof of Lemma A.1 in Barczy and Bezd\'any \cite{BarBez2}).

\begin{Lem}\label{A_1}
Let $(X_k)_{k\in\ZZ_+}$ be a  GWII process defined in \eqref{X_def}.
Assume that $\EE(\xi_{1,1}^2)<\infty$
and $\EE(Y_k)<\infty$, $k\in\ZZ_+$.
Then, for each $k\in\NN$, we have
\begin{enumerate}[label=(\roman*)]
\item
	$\EE(X_k\mid\cF_{k-1})
	=X_k-M_k
	=\EE(\xi_{1,1})X_{k-1}+Y_{k-1}$,
\item
	$\EE(X_k)
	=\EE(\xi_{1,1})\EE(X_{k-1})+\EE(Y_{k-1})
	=\sum_{\ell=1}^k(\EE(\xi_{1,1}))^{k-\ell}\EE(Y_{\ell-1})
	=\sum_{\ell=0}^{k-1}(\EE(\xi_{1,1}))^{k-1-\ell}\EE(Y_{\ell})$,
\item
	$\var(X_k\mid\cF_{k-1})
	=\var(M_k\mid\cF_{k-1})
	=\EE(M_k^2\mid\cF_{k-1})
	=\var(\xi_{1,1})X_{k-1}$,
\item
	$\var(M_k)
	=\EE(M_k^2)
	=\var(\xi_{1,1})\EE(X_{k-1})$.
\end{enumerate}
\end{Lem}

\section{Proof of Theorem \ref{main_3}}\label{Sec_Proof_Thm_main3}

Recall the notation $a_n = \EE(\xi_{1,1}^{(n)})$, $n\in\NN$.
Note that since $a_n\to1$ as $n\to\infty$, we may assume without loss of generality that $a_n>0$, $n\in\NN$.
Recall that, by \eqref{X_alt} and $X^{(n)}_0=0$, $n\in\NN$, we have
\begin{equation}\label{3_cr_X_decomp}
X_\nt^{(n)}=\sum_{j=0}^{\nt-1} a_n^{\nt-1-j}Y_j^{(n)}+\sum_{j=1}^\nt a_n^{\nt-j}M_{j}^{(n)},
 \qquad t\in\RR_+, \qquad n\in\NN.
\end{equation}
Let us consider the decomposition
\begin{align}
\begin{split}\label{X_decomp_3}
\sum_{j=0}^{\nt-1} a_n^{\nt-1-j}Y_j^{(n)}
	=&\sum_{j=0}^{\nt-1} a_n^{\nt-1-j}\EE\left(Y_j^{(n)}\right)
	+\sum_{j=0}^{\nt-1} a_n^{\nt-1-j}\left(Y_j^{(n)}-\EE\left(Y_j^{(n)}\right)\right)
\end{split}
\end{align}
 for all $t\in\RR_+$ and each $n\in\NN$.
By Lemma \ref{A_1} and $X_0^{(n)}=0$, $n\in\NN$,
 the first term on the right hand side of \eqref{X_decomp_3}
 is $\EE(X_\nt^{(n)})$ for all $t\in\RR_+$ and $n\in\NN$.
Thus, by Theorem \ref{exp_thm}, we have
\begin{align*}
\left(n^{-1}\sum_{j=0}^{\nt-1} a_n^{\nt-1-j}\EE\left(m_n^{-1}Y_j^{(n)}\right)\right)_{t\in\RR_+}
	&=\left((nm_n)^{-1}\EE(X_\nt^{(n)})\right)_{t\in\RR_+}\\
	&\lu\left(\int_0^t\ee^{\gamma(t-s)}\beta(s)\,\dd s\right)_{t\in\RR_+} \qquad \text{as $n\to\infty$.}
\end{align*}
Let $T\in\RR_{++}$ be fixed for the rest of this proof.
We check that
 \begin{align}\label{help_1}
 \EE\left(\sup_{t\in[0,T]}\left|(nm_n)^{-1}\sum_{j=0}^{\nt-1} a_n^{\nt-1-j}\left(Y_{j}^{(n)}
	-\EE\left(Y_{j}^{(n)}\right)\right)\right|\right)\to0 \qquad \text{as $n\to\infty$.}
 \end{align}
By the triangle inequality, and using that $\sum_{j=0}^{-1}:=0$, we have that
\begin{align*}
&\sup_{t\in[0,T]}\left|\sum_{j=0}^{\nt-1} a_n^{\nt-1-j}\left(Y_{j}^{(n)}
	-\EE\left(Y_{j}^{(n)}\right)\right)\right|
	=\sup_{t\in[n^{-1},T]}\left|a_n^{\nt-1}\sum_{j=0}^{\nt-1} a_n^{-j}\left(Y_{j}^{(n)}
	-\EE\left(Y_{j}^{(n)}\right)\right)\right|\\
&\qquad\qquad\qquad
	\leq \sup_{t\in[n^{-1},T]} a_n^{\nt-1} \sup_{t\in[n^{-1},T]}\sum_{j=0}^{\nt-1} a_n^{-j}\left|Y_{j}^{(n)}
	-\EE\left(Y_{j}^{(n)}\right)\right|\\
&\qquad\qquad\qquad
	\leq(1\vee a_n^\nT)\sum_{j=0}^{\nT-1} a_n^{-j}\left|Y_{j}^{(n)}
	-\EE\left(Y_{j}^{(n)}\right)\right|, \qquad n\in\NN,\;\; T> n^{-1}.
\end{align*}
Thus, using Jensen's inequality, we get that
\begin{align*}
&\EE\left(\sup_{t\in[0,T]}\left|(nm_n)^{-1}\sum_{j=0}^{\nt-1} a_n^{\nt-1-j}\left(Y_{j}^{(n)}
	-\EE\left(Y_{j}^{(n)}\right)\right)\right|\right)\\
&\qquad\qquad\qquad
	\leq n^{-1}(1\vee a_n^\nT)\sum_{j=0}^{\nT-1} a_n^{-j}\EE\left(\left|m_n^{-1}Y_{j}^{(n)}
	-\EE\left(m_n^{-1}Y_{j}^{(n)}\right)\right|\right)\\
&\qquad\qquad\qquad
	\leq n^{-1}(1\vee a_n^\nT)\sum_{j=0}^{\nT-1} \max_{j\in\{0,\dots,\nT-1\}}\left(a_n^{-j}\EE\left(\left|m_n^{-1}Y_{j}^{(n)}
	-\EE\left(m_n^{-1}Y_{j}^{(n)}\right)\right|\right)\right)
\end{align*}
\begin{align*}
&\qquad\qquad\qquad\qquad
	\leq\frac{\nT}{n}(1\vee a_n^\nT)\left(\max_{j\in\{0,\dots,\nT-1\}}a_n^{-j}\right)\sup_{t\in[0,T]}\EE\left(\left[\left|m_n^{-1}Y_{\nt}^{(n)}
	-\EE\left(m_n^{-1}Y_{\nt}^{(n)}\right)\right|\right]\right)\\
&\qquad\qquad\qquad\qquad
	\leq T(a_n^{-\nT}\vee a_n^\nT)\sup_{t\in[0,T]}
		\EE
			\left(\left|m_n^{-1}Y_{\nt}^{(n)}
			-\EE\left(m_n^{-1}Y_{\nt}^{(n)}\right)\right|\right)\to0 \qquad \text{as $n\to\infty$,}
\end{align*}
where we used assumption \eqref{Y_mom_zero} and the fact
 that $(a_n^{-\nT}\vee a_n^{\nT})\to\ee^{|\gamma|T}<\infty$ as $n\to\infty$.

By \eqref{X_decomp_3}, \eqref{help_1}, and Jacod and Shiryaev \cite[Chapter VI, Lemma 3.31]{JacShi},
 we have that
\[
\left((nm_n)^{-1}\sum_{j=0}^{\nt-1} a_n^{\nt-1-j}Y_j^{(n)}\right)_{t\in\RR_+}
	\distr\left(\int_0^t\ee^{\gamma(t-s)}\beta(s)\,\dd s\right)_{t\in\RR_+} \qquad \text{as $n\to\infty$.}
\]
Therefore, using again Jacod and Shiryaev \cite[Chapter VI, Lemma 3.31]{JacShi} and \eqref{3_cr_X_decomp},
 to prove \eqref{main_3_X_conv}, it is enough to show that
\[
(nm_n)^{-1}\sup_{t\in[0,T]}\left|\sum_{j=1}^\nt a_n^{\nt-j}M_{j}^{(n)}\right|\stoch0 \qquad \text{as $n\to\infty$}
\]
 which can be proved in the same way as (7.5) in
 the proof of Theorem 3.3 in Barczy and Bezd\'any \cite{BarBez2}
 (indeed, in the corresponding proof in \cite{BarBez2}, the hypothesis \ref{A3} is not used).

\section{Proof of Theorem \ref{RahSha_0}  using Theorem \ref{main_3} }\label{RahSha_proof}

Suppose that the assumptions of Theorem \ref{RahSha_0} hold.
Note that \textup{(\textbf{C4})} includes that $R_1(n)\to\infty$ as $n\to\infty$,
 however, in case of $\alpha_1\in\RR_{++}$, it holds automatically, see, e.g., Bingham et al.\ \cite[Proposition 1.5.1]{BinGolTeu}.

First, we check that the conditions of Theorem \ref{RahSha_0} imply that the conditions of Theorem \ref{main_3} are satisfied
 with $m_n:=R_1(n+1)$, $n\in\NN$, and $\beta(t):=t^{\alpha_1}$, $t\in\RR_+$.

The assumption that $\EE((\xi_{1,1}^{(n)})^2)<\infty$, $n\in\NN$,
 yields that our hypothesis \ref{A1} holds,
 and condition \textup{(\textbf{C3})} together with $\EE((\xi_{1,1}^{(n)})^2)<\infty$,
 $n\in\NN$, is the same as our hypothesis \ref{A4}, therefore condition \textup{(a)}
 of Theorem \ref{main_3} holds.
Note that condition \textup{(\textbf{C2})} is the same as our assumption \textup{(b)} of Theorem \ref{main_3}.

We have that the first statement of condition \textup{(\textbf{C1})} can be written as
\[
	\sup_{s\in[0,t]}\left|\EE(m_n^{-1}Y_\ns^{(n)})-m_n^{-1}R_1(\ns+1)\right|\to0 \qquad \text{as $n\to\infty$ for all $t\in\RR_+$.}
\]
Since $\alpha_1\in\RR_{++}$, by Corollary \ref{reg_tf_cor}, we have that
\[
	(m_n^{-1}R_1(\nt+1))_{t\in\RR_+}
		=\left(\frac{R_1(\nt+1)}{R_1(n+1)}\right)_{t\in\RR_+}
		\lu(\beta(t))_{t\in\RR_+} \qquad \text{as $n\to\infty$.}
\]
Thus, using that
\begin{align*}
\sup_{s\in[0,t]}\left|\EE(m_n^{-1}Y_\ns^{(n)})-\beta(s)\right|
	\leq&\sup_{s\in[0,t]}\left|\EE(m_n^{-1}Y_\ns^{(n)})-m_n^{-1}R_1(\ns+1)\right|\\
	&+\sup_{s\in[0,t]}\left|m_n^{-1}R_1(\ns+1)-\beta(s)\right|
\end{align*}
  for all $t\in\RR_+$ and each $n\in\NN$,
 condition \textup{(\textbf{C1})} implies that our condition \eqref{A2_beta_cont} holds
 with the function $\RR_+\ni t\mapsto t^{\alpha_1}=\beta(t)$.
Since $\beta$ is continuous, \eqref{A2_beta_cont} is equivalent to \eqref{A2_beta} by
 part \textup{(i)} of Lemma \ref{Jto_basic}, and thus, taking also into account that $\EE((Y^{(n)}_k)^2)<\infty$, $k\in\ZZ_+$, $n\in\NN$,
 condition \textup{(c)} of Theorem \ref{main_3} holds.

Similarly, the second statement of \textup{(\textbf{C1})} can be written as
\[
	\sup_{s\in[0,t]}\left|(R_2(n+1))^{-1}\var(Y_\ns^{(n)})-(R_2(n+1))^{-1}R_2(\ns+1)\right|\to0 \qquad \text{as $n\to\infty$  for all $t\in\RR_+$.}
\]
Since $\alpha_2\in\RR_{++}$, by Corollary \ref{reg_tf_cor}, we have that
\[
	\left(\frac{R_2(\nt+1)}{R_2(n+1)}\right)_{t\in\RR_+}\lu(t^{\alpha_2})_{t\in\RR_+} \qquad \text{as $n\to\infty$.}
\]
Thus, for all $T\in\RR_{++}$, we get that
\begin{align*}
\sup_{t\in[0,T]}\left|(R_2(n+1))^{-1}\var(Y_\nt^{(n)})-t^{\alpha_2}\right|
	\leq&\sup_{t\in[0,T]}\left|(R_2(n+1))^{-1}\var(Y_\nt^{(n)})-(R_2(n+1))^{-1}R_2(\nt+1)\right|\\
	&+\sup_{t\in[0,T]}\left|(R_2(n+1))^{-1}R_2(\nt+1)-t^{\alpha_2}\right|\to0
\end{align*}
as $n\to\infty$, that is,
\[
	\left((R_2(n+1))^{-1}\var(Y_\nt^{(n)})\right)_{t\in\RR_+}\lu(t^{\alpha_2})_{t\in\RR_+} \qquad \text{as $n\to\infty$.}
\]
Using that $\frac{R_2(n+1)}{m_n^2}=\frac{R_2(n+1)}{(R_1(n+1))^2}\to0$ as $n\to\infty$ (see condition \textup{(\textbf{C4})})
 and part \textup{(ii)} of Lemma \ref{Jto_basic}, we get that
\begin{equation}\label{var_0_helper}
	\left(\var(m_n^{-1}Y_\nt^{(n)})\right)_{t\in\RR_+}
		=	\left(\frac{R_2(n+1)}{m_n^2}(R_2(n+1))^{-1}\var(Y_\nt^{(n)})\right)_{t\in\RR_+}
		\lu(0)_{t\in\RR_+}
\end{equation}
as $n\to\infty$.
By the Jensen inequality, we have that
\[
\left(\EE\left[\left|m_n^{-1}Y_\nt^{(n)}-\EE\left(m_n^{-1}Y_\nt^{(n)}\right)\right|\right]\right)^2
	\leq \var(m_n^{-1}Y_\nt^{(n)}), \qquad t\in\RR_+, \qquad n\in\NN,
\]
and thus, by \eqref{var_0_helper}, we get that
\[
	\left(\left(\EE\left[\left|m_n^{-1}Y_\nt^{(n)}-\EE\left(m_n^{-1}Y_\nt^{(n)}\right)\right|\right]\right)^2\right)_{t\in\RR_+}
		\lu(0)_{t\in\RR_+} \qquad \text{as $n\to\infty$,}
\]
 which yields that condition \textup{(d)} of Theorem \ref{main_3} holds.
Thus, all assumptions of Theorem \ref{main_3} hold.

Finally, we show that the statements of Theorem \ref{RahSha_0} follow from Theorem \ref{main_3} as well.
Note that the conditions of Theorem \ref{main_3} imply that the conditions
 of Theorem \ref{exp_thm} hold.
Hence, by Theorem \ref{exp_thm},
we have that 
 \begin{equation}\label{help_5}
	\left(
	(nm_n)^{-1}\EE\left(
	X_\nt^{(n)}\right)\right)_{t\in\RR_+}
		\lu
	\left(
	\int_0^t\ee^{\gamma(t-s)}\beta(s)\,\dd s\right)_{t\in\RR_+}
	 \qquad \text{as $n\to\infty$.}
	\end{equation}
Note that \eqref{help_5} is nothing else but part (a) of Lemma 2 in Rahimov \cite{Rah1}, see also the first assertion of Lemma 2 in Rahimov and Sharipov \cite{RahSha}.
In particular, by \eqref{help_5}, we get
\begin{align}\label{help_4}
(nm_n)^{-1}\EE(X_n^{(n)})=(nm_n)^{-1}\EE(X_{\lfloor n\cdot 1\rfloor}^{(n)})\to\int_0^1\ee^{\gamma(1-s)}\beta(s)\,\dd s
	\qquad \text{as $n\to\infty$,}
\end{align}
 where the limit is positive (since $\beta(t)=t^{\alpha_1}$, $t\in\RR_+$),
 and thus $\EE(X_n^{(n)})>0$ for sufficiently large $n\in\NN$.
In what follows, we may assume without loss of generality that $\EE(X_n^{(n)})>0$, $n\in\NN$.
By Theorem \ref{main_3}, we have that
\begin{equation}\label{RahSha_conv}
\left((nm_n)^{-1}X_\nt^{(n)}\right)_{t\in\RR_+}
	\distr
	\left(\int_0^t\ee^{\gamma(t-s)}\beta(s)\,\dd s\right)_{t\in\RR_+} \qquad \text{as $n\to\infty$.}
\end{equation}
Next, we apply Lemma \ref{Conv2Funct} with $d=q=1$ and the following choices:
\begin{itemize}
\item
	$\Phi(f):=\left(\int_0^1\ee^{\gamma(1-s)}\beta(s)\,\dd s\right)^{-1}f$
		 for $f\in\DD(\RR_+,\RR)$,
\item
	$\Phi_n(f):=\left((nm_n)^{-1}\EE(X_n^{(n)})\right)^{-1}f$
		 for $f\in\DD(\RR_+,\RR)$ and $n\in\NN$,
\item
	$C:=\CC(\RR_+,\RR)\in\cB(\DD(\RR_+,\RR))$,
\item 
	$(\cU_t)_{t\in\RR_+}:=\left(\int_0^t\ee^{\gamma(t-s)}\beta(s)\,\dd s\right)_{t\in\RR_+}$,
\item 
	$(\cU_t^{(n)})_{t\in\RR_+}:=\left((nm_n)^{-1}X_\nt^{(n)}\right)_{t\in\RR_+}$, $n\in\NN$.
\end{itemize}
Since the mapping $\DD(\RR_+,\RR)\ni f\mapsto c^{-1}f\in\DD(\RR_+,\RR)$ is 
 continuous (in particular, Borel measurable) for all $c\in\RR_{++}$ 
  (following, e.g., from Lemma \ref{cont_cont}),
 we have that the mappings $\Phi$ and $\Phi_n$, $n\in\NN$, are Borel measurable.
By \eqref{help_4} and part \textup{(ii)} of Lemma \ref{Jto_basic}, the mappings $\Phi$ and $\Phi_n$, $n\in\NN$,
 satisfy the conditions of Lemma \ref{Conv2Funct} with the set $C$.
The stochastic processes $\cU^{(n)}$, $n\in\NN$, are $\RR_+$-valued and have c\`adl\`ag paths,
 and, by \eqref{RahSha_conv}, $\cU^{(n)}\distr\cU$ as $n\to\infty$ with $\PP(\cU\in C)=1$.
Thus, we can apply Lemma \ref{Conv2Funct} to obtain that
\[
  \left(\EE(X_n^{(n)})^{-1} X_{\nt}^{(n)}
   \right)_{t\in\RR_+}
   =  \left(
           \frac{(nm_n)^{-1}X_\nt^{(n)}}{(nm_n)^{-1}\EE(X_n^{(n)})} \\
   \right)_{t\in\RR_+}
  \distr 
   \left(
           \frac{\int_0^t\ee^{\gamma(t-s)}\beta(s)\,\dd s}{\int_0^1\ee^{\gamma(1-s)}\beta(s)\,\dd s}
   \right)_{t\in\RR_+} 
\]
 as $n\to\infty$, as desired.

\vspace*{5mm}

\appendix

\vspace*{5mm}

\noindent{\bf\Large Appendices}

\section{Locally uniform convergence and convergence in Skorokhod $J_1$ topology}
\label{GWII_moments_CMT}

In this appendix, we recall some results regarding locally uniform convergence,
 convergence in the Skorokhod $J_1$ topology, and a version of the continuous mapping theorem.

Recall that for $d\in\NN$, $\DD(\RR_+,\RR^d)$ and $\CC(\RR_+,\RR^d)$
 denote the set of $\RR^d$-valued c\`adl\`ag and continuous functions, respectively.
Let $\cB(\DD(\RR_+, \RR^d))$ denote the Borel $\sigma$-algebra on
 $\DD(\RR_+, \RR^d)$ for the metric defined in Jacod and Shiryaev
 \cite[Chapter VI, (1.26)]{JacShi} (see also Billingsley \cite[Chapter 3, Section 16]{Bil}).
With this metric $\DD(\RR_+, \RR^d)$ is a
 complete and separable metric space and the topology induced by this metric is
 the so-called Skorokhod $J_1$ topology.
Note that $\CC(\RR_+,\RR^d)\in\cB(\DD(\RR_+,\RR^d))$, see, e.g., Ethier and Kurtz \cite[Problem 3.11.25]{EthKur}.

First, we collect some simple facts about convergence in the Skorokhod $J_1$ topology
 and locally uniform convergence
 (for references, see Barczy and Bezd\'any \cite[Lemma B.1]{BarBez2}).

\begin{Lem}\label{Jto_basic}
Let $d,q\in\NN$, and let $f\in\DD(\RR_+,\RR^d)$, $f_n\in\DD(\RR_+,\RR^d)$, $n\in\NN$,
 and $g\in\DD(\RR_+,\RR^q)$, $g_n\in\DD(\RR_+,\RR^q)$, $n\in\NN$.
\begin{enumerate}[label=(\roman*)]
\item
	If $f$ is continuous, then $f_n\Jto f$ as $n\to\infty$ if and only if $f_n\lu f$ as $n\to\infty$.
\item
	Suppose that $f_n\lu f$ and $g_n\lu g$ as $n\to\infty$.
	Then $[f_n, g_n]^\top\lu[f,g]^{\top}$ as $n\to\infty$.
	Further, if $q=d$, then $f_n+g_n\lu f+g$ as $n\to\infty$,
	 and if $q=1$ and $g$ is continuous, then $f_n\cdot g_n\lu f\cdot g$ as $n\to\infty$.
\item
	Suppose that $f_n\Jtoo{d} f$ and $g_n\Jtoo{q} g$ as $n\to\infty$, and $g\in\CC(\RR_+,\RR^q)$.
	Then $[f_n, g_n]^\top\Jtoo{d+q}[f,g]^{\top}$ as $n\to\infty$.
	Further, if $q=d$, then $f_n+g_n\Jtoo{d}f+g$ as $n\to\infty$,
	 and if $q=1$, then $f_n\cdot g_n\Jtoo{d}f\cdot g$ as $n\to\infty$.
\end{enumerate}
\end{Lem}

Next, we recall a result concerning the composition of a c\`adl\`ag function
  with a continuous function (see, e.g., Ethier and Kurtz \cite[Problem 3.11.13]{EthKur}).

\begin{Lem}\label{cont_cont}
Let $d,q\in\NN$, and let $\varphi:\RR^d\to\RR^q$ be continuous.
Then the mapping
 $\DD(\RR_+,\RR^d)\ni f\mapsto \varphi\circ f=\left(\varphi(f(t))\right)_{t\in\RR_+}\in\DD(\RR_+,\RR^q)$
 is continuous.
\end{Lem}

For $\RR^d$-valued stochastic processes $(\bcY_t)_{t \in \RR_+}$ and
 $(\bcY^{(n)}_t)_{t \in \RR_+}$, $n \in \NN$, with c\`adl\`ag paths, we write
 $\bcY^{(n)} \distr \bcY$ as $n\to\infty$ if the distribution of $\bcY^{(n)}$ on the
 space $(\DD(\RR_+,  \RR^d), \cB(\DD(\RR_+, \RR^d)))$ converges weakly to the
 distribution of $\bcY$ on the space
 $(\DD(\RR_+,  \RR^d), \cB(\DD(\RR_+, \RR^d)))$ as $n \to \infty$.
If $\xi$ and $\xi_n$, $n \in \NN$, are random elements with values in a metric space $(E,d)$,
 then we denote by $\xi_n \distr \xi$ as $n\to\infty$ the weak convergence of the
 distribution of $\xi_n$ on the space $(E, \cB(E))$ towards the
 distribution of $\xi$ on the space $(E, \cB(E))$ as $n \to \infty$,
 where $\cB(E)$ denotes the Borel $\sigma$-algebra on $E$ induced by
 the given metric $d$.

Finally, we formulate a version of
 the continuous mapping theorem, which easily follows from Kallenberg \cite[Theorem 3.27]{Kal}.
Given $d,q\in\NN$, for Borel measurable mappings $\Phi : \DD(\RR_+, \RR^d) \to \DD(\RR_+, \RR^q)$
 and $\Phi_n : \DD(\RR_+, \RR^d) \to \DD(\RR_+, \RR^q)$, $n \in \NN$,
 we will denote by $C_{\Phi, (\Phi_n)_{n \in \NN}}$ the set of all functions
 $f \in \CC(\RR_+, \RR^d)$ for which $\Phi_n(f_n) \lu \Phi(f)$ as $n\to\infty$ whenever
 $f_n \lu f$ as $n\to\infty$ with $f_n \in \DD(\RR_+, \RR^d)$, $n \in \NN$.

\begin{Lem}\label{Conv2Funct}
Let $d,\,q\in\NN$.
Let $\Phi : \DD(\RR_+, \RR^d) \to \DD(\RR_+, \RR^q)$ and
  $\Phi_n : \DD(\RR_+, \RR^d) \to \DD(\RR_+, \RR^q)$, $n \in \NN$, be Borel measurable mappings.
Let $(\bcU_t)_{t \in \RR_+}$ and $(\bcU^{(n)}_t)_{t \in \RR_+}$, $n \in \NN$,
 be $\RR^d$-valued stochastic processes with c\`adl\`ag paths such that
 $\bcU^{(n)} \distr \bcU$ as $n\to\infty$.
Suppose that there exists $C\in\cB(\DD(\RR_+,\RR^d))$
 such that $\PP(\bcU \in C)=1$ and $C\subset C_{\Phi,(\Phi_n)_{n\in\NN}}$.
Then $\Phi_n(\bcU^{(n)})\distr \Phi(\bcU)$ as $n\to\infty$.
\end{Lem}

\section{Regularly varying functions}\label{App_reg}

In this appendix, we present some simple auxiliary results  on regularly varying functions,
 and, for completeness, we provide proofs as well.
First, we recall the notions of slowly varying and regularly varying functions (see, e.g., Bingham et al.\ \cite{BinGolTeu}).

\begin{Def}
A measurable function $f: \RR_{++} \to \RR_{++}$ is called regularly varying at infinity with
index $\rho \in \RR$ if
\[
\lim_{x\to\infty} \frac{f(xt)}{f(x)} = t^\rho, \qquad t \in \RR_{++}.
\]
In case of $\rho = 0$, we call $f$ slowly varying at infinity.
\end{Def}

The following lemma is a simple corollary of the uniform convergence theorem
  (see, e.g., Bingham et al.\ \cite[Theorem 1.5.2]{BinGolTeu}),
 but we state and prove it for the sake of completeness.

\begin{Lem}\label{seq_reg}
Let $f:\RR_{++}\to\RR_{++}$ be a function regularly varying at infinity with index $\rho\in\RR$.
If $t\in\RR_{++}$ and $t_n,x_n\in\RR_{++}$, $n\in\NN$, are
 such that $t_n\to t$ as $n\to\infty$ and $x_n\to\infty$ as $n\to\infty$, then
\[
\lim_{n\to\infty}\frac{f(x_nt_n)}{f(x_n)}=t^\rho.
\]
\end{Lem}
\noindent\textbf{Proof.}
Let $f:\RR_{++}\to\RR_{++}$ be a function regularly varying at infinity with index $\rho\in\RR$,
 and let $t\in\RR_{++}$ and $t_n,x_n\in\RR_{++}$, $n\in\NN$,
 be such that $x_n\to\infty$ as $n\to\infty$ and $t_n\to t$ as $n\to\infty$.
Since $f$ is regularly varying at infinity with index $\rho$,
 there exists a function $\ell:\RR_{++}\to\RR_{++}$ slowly varying at infinity
 such that $f(s)=s^\rho\ell(s)$, $s\in\RR_{++}$.
We have that
\[
\left|\frac{f(x_nt_n)}{f(x_n)}-t^\rho\right|
	=\left|t_n^\rho\frac{\ell(x_nt_n)}{\ell(x_n)}-t^\rho\right|
	\leq t_n^\rho\left|\frac{\ell(x_nt_n)}{\ell(x_n)}-1\right|
	+\left|t_n^\rho-t^\rho\right|, \qquad n\in\NN.
\]
The second term on the right hand side of the inequality tends to $0$
 as $n\to\infty$ due to the continuity of the power function on $\RR_{++}$.
Note that, since $t\in\RR_{++}$, $[t/2,2t]$ is a compact subset of $\RR_{++}$.
Since $t_n\to t$ as $n\to\infty$, we have that $t_n\in[t/2,2t]$ for large enough $n\in\NN$,
 and thus
\[
\lim_{n\to\infty}t_n^\rho\left|\frac{\ell(x_nt_n)}{\ell(x_n)}-1\right|
	\leq\lim_{n\to\infty}t_n^\rho\sup_{s\in[t/2,2t]}\left|\frac{\ell(x_n s)}{\ell(x_n)}-1\right|=t^\rho0=0,
\]
 where, since $x_n\to\infty$ as $n\to\infty$,
 the penultimate equality follows from the uniform convergence theorem
 (see, e.g., Bingham et al.\ \cite[Theorem 1.5.2]{BinGolTeu})
 and the fact that $t_n\to t$ as $n\to\infty$.
\proofend

\begin{Lem}\label{reg_tf}
Let $f:\RR_{++}\to\RR_{++}$ be a function
 regularly varying at infinity with index $\rho\in\RR$.
Then the function $g:\RR_+\to\RR_{++}$, $g(t):=f(\lfloor t\rfloor+1)$, $t\in\RR_+$,
 is bounded on $[0,T]$ for all $T\in\RR_{++}$, and
 the restriction of the function $g$ onto $\RR_{++}$ is regularly varying at infinity with index $\rho$.
\end{Lem}

\noindent \textbf{Proof.}
For all $T\in\RR_{++}$, we have that
\[
\sup_{t\in[0,T]}|g(t)|
	=\sup_{t\in[0,T]}|f(\lfloor t\rfloor+1)|
	=\max_{j\in\{1,\dots,\lfloor T\rfloor+1\}}|f(j)|<\infty,
\]
as desired.
Let $t\in\RR_{++}$ and $x_n\in\RR_{++}$, $n\in\NN$, be such that $x_n\to\infty$ as $n\to\infty$.
Then we have that
\[
\frac{g(x_nt)}{g(x_n)}
	=\frac{f(\lfloor x_nt\rfloor+1)}{f(\lfloor x_n\rfloor+1)}
	=\frac{f\left((\lfloor x_n\rfloor+1)\frac{\lfloor x_nt\rfloor+1}{\lfloor x_n\rfloor+1}\right)}{f(\lfloor x_n\rfloor+1)}
	\to t^{\rho} \qquad \text{as $n\to\infty$,}
\]
where the convergence follows from Lemma \ref{seq_reg}
 and the facts that $\lfloor x_n\rfloor+1\to\infty$ as $n\to\infty$
 and $\frac{\lfloor x_nt\rfloor+1}{\lfloor x_n\rfloor+1}\to t$ as $n\to\infty$.
\proofend

\begin{Cor}\label{reg_tf_cor}
Let $f:\RR_{++}\to\RR_{++}$ be a function
 regularly varying at infinity with index $\rho\in\RR_{++}$.
Then
\[
\left(\frac{f(\nt+1)}{f(n+1)}\right)_{t\in\RR_+}
	\lu\left(t^\rho\right)_{t\in\RR_+}
	\qquad \text{as $n\to\infty$.}
\]
\end{Cor}
\noindent\textbf{Proof.}
First, note that, by Lemma \ref{reg_tf},
 the function $\RR_{+}\ni t\mapsto f(\lfloor t\rfloor+1)\in\RR_{++}$
 is bounded on $[0,T]$ for all $T\in\RR_{++}$,
 and its restriction onto $\RR_{++}$ is
 regularly varying at infinity with index $\rho$.
Let $T\in\RR_{++}$ be fixed. Then we have that
\begin{equation*}
\sup_{t\in[0,T]}\left|\frac{f(\nt+1)}{f(n+1)}-t^\rho\right|
	\leq\left|\frac{f(1)}{f(n+1)}\right|+\sup_{t\in(0,T]}\left|\frac{f(\nt+1)}{f(n+1)}-t^\rho\right|, \qquad n\in\NN.
\end{equation*}
Since $\rho\in\RR_{++}$, by Bingham et al.\ \cite[Proposition 1.5.1]{BinGolTeu},
 we have that $f(n)\to\infty$ as $n\to\infty$,
 and thus the first term on the right hand side of the above inequality tends to $0$ as $n\to\infty$.
Since $\rho\in\RR_{++}$ and the function $\RR_{++}\ni t\mapsto f(\lfloor t\rfloor+1)\in\RR_{++}$
 is bounded on $(0,T]$,
 the second term on the right hand side of the above inequality tends to $0$ as $n\to\infty$
 by the uniform convergence theorem (see, e.g., Bingham et al.\ \cite[Theorem 1.5.2]{BinGolTeu}).
\proofend

Finally, we point out that Corollary \ref{reg_tf_cor} does not hold in general when $\rho=0$.
\begin{Rem}\label{Rem_reg}
Let $f:\RR_{++}\to\RR_{++}$, $f(x):=\ln(x+1)$, $x\in\RR_{++}$, which is slowly varying at infinity.
Supposing that Corollary \ref{reg_tf_cor} holds
	in the case $\rho=0$ as well, we would have that
\begin{equation}\label{B_rem_conv}
	\left(\frac{f(\nt+1)}{f(n+1)}\right)_{t\in\RR_+}
	=\left(\frac{\ln(\nt+2)}{\ln(n+2)}\right)_{t\in\RR_+}\lu(t^0)_{t\in\RR_+}=(1)_{t\in\RR_+}
	 \qquad \text{as $n\to\infty$.}
\end{equation}
For each $n\in\NN$, we have that $n^{-1}\in(0,1]$, and hence
\[
\sup_{t\in[0,1]}\left|\frac{f(\nt+1)}{f(n+1)}-t^0\right|
	\geq \left|\frac{\ln(\lfloor n\cdot n^{-1}\rfloor+2)}{\ln(n+2)}-1\right|
	=\left|\frac{\ln(3)}{\ln(n+2)}-1\right|
	\to1 \qquad \text{as $n\to\infty$,}
\]
yielding that \eqref{B_rem_conv} cannot hold.
This shows that Corollary \ref{reg_tf_cor} does not hold in general when $\rho=0$.
\proofend
\end{Rem}

%\section*{Acknowledgements}

\section*{Statements and Declarations}

The authors declare that they have no known competing financial interests
or personal relationships that could have appeared to influence the work presented in this paper. 
Data sharing is not applicable to this article as no datasets were generated or analyzed during the current study.

\end{document}